\documentclass[english,12pt]{smfart}
\usepackage{amssymb}
\usepackage{amsfonts}
\usepackage{amscd}
\usepackage[T1]{fontenc}

\usepackage{color}

\usepackage{mathrsfs} 

\newcommand{\CO}[1][]{\mathcal{C}_{\cO#1}}

\newcommand{\codim}{\operatorname{codim}}

\newcommand{\grad}{\operatorname{grad}}

\newcommand{\Cont}{\mathrm{Cont}}

\def\11{{\mathbf 1}}
\def\AA{{\mathbb A}}

\def\CC{{\mathbb C}}

\def\FF{{\mathbb F}}

\def\NN{{\mathbb N}}

\def\QQ{{\mathbb Q}}

\def\ZZ{{\mathbb Z}}

\def\cM{{\mathcal M}}

\def\cO{{\mathcal O}}

\def\llp{\mathopen{(\!(}}

\def\rrp{\mathopen{)\!)}}

\newtheorem{thm}[subsection]{Theorem}
\newtheorem{lem}[subsection]{Lemma}

\newtheorem{prop}[subsection]{Proposition}
\newtheorem{conj}[subsection]{Conjecture}

\theoremstyle{definition}
\newtheorem{defn}[subsection]{Definition}

\newtheorem{def-prop}[subsection]{Proposition-Definition}
\newtheorem{def-thm}[subsection]{Theorem-Definition}
\newtheorem{def-lem}[subsection]{Lemma-Definition}

\theoremstyle{remark}
\newtheorem{remark}[subsection]{Remark}

\theoremstyle{plain}

\numberwithin{equation}{subsection}

\renewcommand{\theequation}{\thesubsection.\arabic{equation}}

  {\par\medskip\noindent #1\par\begingroup%
    \advance\leftskip by 1em\advance\rightskip by 1em}%
  {\par\endgroup}

\newcommand{\ord}{\operatorname{ord}}

\title[Exponential sums and log-canonical thresholds]{Bounds for
$p$-adic exponential sums and log-canonical thresholds} 

\author{Raf Cluckers}
\address{Universit\'e Lille 1, 
Laboratoire Painlev\'e,
 CNRS - UMR 8524, Cit\'e Scientifique, 59655
Villeneuve d'Ascq Cedex, France, and,
KU Leuven, Department of Mathematics,
Celestijnenlaan 200B, B-3001 Leu\-ven, Bel\-gium\\}
\email{Raf.Cluckers@math.univ-lille1.fr}
\urladdr{http://math.univ-lille1.fr/$\sim$cluckers}

\author{Willem Veys}
\address{KU Leuven, Department of Mathematics,
Celestijnenlaan 200B, B-3001 Leu\-ven, Bel\-gium}
\email{wim.veys@wis.kuleuven.be}
\urladdr{http://wis.kuleuven.be/algebra/veys.htm}

\begin{abstract}
We propose a conjecture for exponential sums which generalizes both a conjecture by Igusa and a local variant by Denef and Sperber, in particular, it is without the homogeneity condition on the polynomial in the phase, and with new predicted uniform behavior. The exponential sums have summation sets consisting of integers modulo $p^m$ lying $p$-adically close to $y$, and the proposed bounds are uniform in
$p$, $y$, and $m$. We give evidence for the conjecture, by showing uniform bounds in $p$, $y$, and in some values for $m$. On the way, we prove new bounds for log-canonical thresholds which are closely related to the bounds predicted by the conjecture.
\end{abstract}



\begin{document}

\maketitle

\section{Introduction and main results}\label{intro}

We introduce a generalization of a conjecture by Igusa \cite[page 2]{Igusa3} (and of a variant by Denef and Sperber \cite[page 2]{DenSper}), which Igusa related to integrability properties over the ad\`eles and to an ad\`elic Poisson summation formula in \cite[Chapter 4]{Igusa3}. We give evidence for this conjecture, which is also new evidence for the original conjectures of \cite{Igusa3} and \cite{DenSper}. The conjecture is about upper bounds for exponential sums of the form
$$
\sum_{x\in \{1,\ldots,N\}^n} \exp ( 2\pi i \frac{F(x)}{N} )
$$
for general polynomials $F$ over $\ZZ$ in $n$ variables, expressed in terms of $N$ and holding for all squarefull integers $N$. It is most conveniently expressed when $N$ is a power of a prime number, the power being at least $2$, and can be studied via a local variant, see the sums $S$ and $S_y$ below and Conjecture \ref{new}. A variant over number fields is given in Section \ref{sec:K}.

Let us fix a nonconstant polynomial $F$ in $n$ variables over $\ZZ$.
Consider, for any integer $m>1 $ and any prime number $p$, the
exponential sum
$$
S(F,p,m) :=  p^{-mn}\cdot\sum_{x\in (\ZZ/p^m\ZZ)^n} \exp ( 2\pi i \frac{F(x)}{p^m} ),
$$
and, for any $y\in\ZZ^n$, its local version
$$
S_y(F,p,m) :=  p^{-mn}\cdot\sum_{x\in y + (p\ZZ/p^m\ZZ)^n} \exp ( 2\pi i \frac{F(x)}{p^m} ),
$$
where
$$
y + (p\ZZ/p^m\ZZ)^n = \{x\in (\ZZ/p^m\ZZ)^n\mid x_i\equiv y_i\bmod (p)\mbox{ for each } i\}.
$$


Our conjectured bounds for
the above sums in terms of $p$, $m,$
and $y$ (and our evidence for these bounds) will involve log-canonical
thresholds, but a stronger formulation in terms of the motivic
oscillation index of \cite{Cigumodp} or the complex oscillation index of \cite[13.1.5]{Arnold.G.V.II} would also make sense and would in fact sometimes be sharper. For any
field $k$ of characteristic zero, a polynomial $f\in k[x]=
k[x_1,\dots,x_n]$ and a point $y\in k^n$ satisfying $f(y)=0$, we
write $c_y(f)$ to denote the log-canonical threshold of $f$ at $y$
(see Definition \ref{c(f)} below), and $c(f)$ for the log canonical
threshold of $f$, being the minimum of the $c_y(f)$ when $y$ runs
over all points in $\overline k^n$ satisfying $f(y)=0$, where
$\overline k$ is an algebraic closure of $k$.
 Let us fix some more notation.

\begin{defn}\label{log}
Let $a(F)$ be the minimum, over all $b\in \CC$, of the log-canonical thresholds of the polynomials $F(x)-b$.
Further, for $y\in\ZZ^n$, let $a_{y,p}(F)$ be the minimum of the
log-canonical thresholds at $y'$ of the polynomials $F(x)-F(y')$,
where the minimum is taken over all $y'\in y+(p\ZZ_p)^n$.  Note that $a(F)\leq a_{y,p}(F)$ for each $p$ and $y$. 

\end{defn}

Now we can state our generalization of the conjectures by Igusa and by Denef and Sperber.

\begin{conj}\label{new}
There exists a function $L_F:\NN\to \NN$ with $L_F(m) \ll m^{n-1}$ such that for all primes $p$, all $m\geq 2$, and all $y \in \ZZ^n$, one has
\begin{equation}\label{1}
| S(F,p,m) |_\CC \leq L_F(m) p^{-m a(F)}
\end{equation}
and
\begin{equation}\label{1y}
| S_y(F,p,m) |_\CC \leq L_F(m) p^{-m a_{y,p}(F)},
\end{equation}
where $|\cdot|_\CC$ is the complex modulus.
\end{conj}


Under some extra conditions that were introduced by Igusa for
reasons of his application to ad\`elic integrability but that we believe are irrelevant for
bounding the above sums, he conjectured in the introduction of
\cite{Igusa3} that (\ref{1}) holds for all homogeneous $F$ and all
$m\geq 1$. We believe that focusing on $m$ at least $2$ allows one
to remove the homogeneity condition, and we give evidence below. The
bounds (\ref{1}) (with the log-canonical threshold, resp.~the
variant with the motivic oscillation index of \cite{Cigumodp} in the exponent) imply
Igusa's original conjecture (with the log-canonical threshold,
resp.~his proposed candidate oscillation indices in the
exponent), including the case $m=1$, by \cite{Cigumodp}. Indeed, the
case $m=1$ of Igusa's conjecture (for homogeneous $F$) is known by \cite{Cigumodp} for
any of these exponents. The estimates (\ref{1}) of the conjecture
yield a criterion to show ad\`elic $L^q$-integrability for an ad\`elic
function related to $S(F,p,m)$, with a simple lower bound on $q$
based on the exponent $a(F)$, as noted by Igusa in \cite[Chapter 4]{Igusa3}.
Denef and Sperber \cite{DenSper} conjectured the local variant
(\ref{1y}) for $y=0$, thus without uniformity in $y$. Both
inequalities, namely the global (\ref{1}) and the local but uniform
(\ref{1y}), seem closely related.

We prove Conjecture \ref{new} for $m$ up to the value $4$ in general, and, more specifically, for $m$ up to some value related to orders of vanishing, defined as follows. 

\begin{defn}\label{r}
Let $r$ be the minimum of the order of vanishing of the functions
$x\mapsto F(x)- b $ at the singular points in $\CC^n$ of $F=b$,
i.e., the minimum of the multiplicities of the singular points of
the hypersurfaces $F=b$, where $b$ runs over $\CC$.
Here we consider the minimum over the empty set to be $+\infty$.
Further, for $y\in \ZZ^n$, let $r_{y,p}$ be the minimum of the order
of vanishing of the functions $x\mapsto F(x)-F(y')$ at $y'$, where
$y'$ runs only over singular points in the $p$-adic neighbourhood
$y+(p\ZZ_p)^n$ for which moreover $c_{y'}(F-F(y')) = a_{y,p}(F)$.

\end{defn}

Note that by definition $r_{y,p}\geq r\geq 2$ and $1\geq a_{y,p}(F)\geq a(F)\geq 0$.
With notation as introduced above and with $+\infty+a=+\infty$ for any real $a$, we can now state our main result as evidence for Conjecture \ref{new}.

\begin{thm}\label{rplus1}
There exists a constant $L_F$ such that, for all prime numbers $p$, all $y\in \ZZ^n$, and all $m$ with $2\leq m\leq r+2$, resp.~with $2\leq m\leq r_{y,p}+2$, one has
\begin{equation}\label{b}
| S(F,p,m) |_\CC \leq L_F p^{- m a(F)},
\end{equation}
resp.
\begin{equation}\label{by}
| S_y(F,p,m) |_\CC \leq L_F p^{- m a_{y,p}(F)}.
\end{equation}
\end{thm}

Theorem \ref{rplus1} is proved using new inequalities for log-canonical thresholds and by reducing to finite field exponential sums for which bounds by Katz \cite{Katz} can be used, see Lemma \ref{lemKa}.
In Section \ref{sec:K}, we explain analogues over finite field extensions of $\QQ_p$ and $\FF_p\llp t \rrp$, for large primes $p$.

\subsection*{}
Let us now explain the bounds on log-canonical thresholds related to the conjecture.
Let $f$ be a nonconstant polynomial over $\CC$ in the variables $x=(x_1,\ldots,x_n)$, and
write
\begin{equation}\label{ffr}
f = \sum_{i\geq r} f_i,
\end{equation}
with $f_i$ either identically zero or homogeneous and of degree $i$, and where $f_r$ is nonzero for some $r\geq 2$.
As before, write $c_0(f)$ for the log-canonical threshold of $f$ at zero.
If $f$ is non-reduced at zero (that is, $g^2$ divides $f$ for some polynomial $g$ which vanishes at $0$), then one knows that
\begin{equation}\label{trivial0}
c_0(f)\leq \frac{1}{2}.
\end{equation}
In any case one has (see Section 8 of \cite{Kollar})
\begin{equation}\label{trivial}
c_0(f) \leq \frac{n}{r}.
\end{equation}
 The following inequalities can be considered as a certain combination of the above two (quite obvious) inequalities, but with the non-reducedness assumption on $f_r$ instead of on $f$.
\begin{lem}\label{geomc}
Suppose that $g^2$ divides $f_r$ for some nonconstant polynomial $g$.
Then one has the inequality
\begin{equation}\label{c_0-1}
(r+1) c_0(f) \leq n + \frac{1}{2}.
\end{equation}
If moreover $g$ divides $f_{r+1}$ (this includes the case $f_{r+1}$ identically zero), then
\begin{equation}\label{c_0-2}
(r+2) c_0(f) \leq n + 1.
\end{equation}
\end{lem}

Lemma \ref{geomc} will be obtained as a corollary of the following sharper and unconditional bounds, which we think are of independent interest.

\begin{prop}\label{geomcfr}
With notation from (\ref{ffr}), one has
\begin{equation}\label{c_0}
(r+1) c_0(f) \leq n + c(f_r).
\end{equation}
\end{prop}

One should compare (\ref{c_0})
with the bound $|c_0(f) - c(f_r)| \leq n/(r+1)$ from Proposition 8.19 of \cite{Kollar}.
A generalization of Proposition \ref{geomcfr}, with a bound
for $(e+1)c_0(f)$ for arbitrary $e>0$, is given in Section \ref{recursive}, see Theorem \ref{geomcfe}.
By combining Lemma \ref{geomc} with results from \cite{DimcaMST}, we obtain  global variants.
\begin{prop}\label{geom2}
Let $r>1$ be an integer and let $f$ be a polynomial in $n$ variables over $\CC$.
Suppose, for $y$ running over an irreducible $d$-dimensional variety $Y\subset \CC^n$, that $f$ vanishes with order at least $r$ at $y$.
For $y\in Y$, let us write $f_y(x)$ for the polynomial $f(x+y)$ in the variables $x$, and $f_y = \sum_{i\geq r} f_{y,i}$ with $f_{y,i}$ either identically zero or homogeneous and of degree $i$.
Then one has
\begin{equation}\label{c_0b}
r c_y(f) \leq n - d
\end{equation}
and, for a generic $y\in Y$,
\begin{equation}\label{c_0-r}
(r+1) c_y(f) \leq n - d + c(f_{y,r}).
\end{equation}
In particular, for a generic $y\in Y$, if $f_{y,r}$ is non-reduced, then
\begin{equation}\label{c_0-1b}
(r+1) c_y(f) \leq n - d + \frac{1}{2}.
\end{equation}
If, for a generic $y\in Y$, there is a non-constant polynomial $g_y$ which divides $f_{y,r+1}$ 
and such that $g_y^2$ divides $f_{y,r}$, then one further has
\begin{equation}\label{c_0-2b}
(r+2) c_y(f) \leq n -d + 1.
\end{equation}
\end{prop}

The proofs of Theorem \ref{rplus1}, Proposition \ref{geomcfr}, Lemma
\ref{geomc} and the global variants are given in Section
\ref{proofs}.

\subsection{Some context and notation}\label{not}

Conjecture \ref{new} is known when the implied constant is allowed
to depend on the prime number $p$, see \cite{Igusa3} and
\cite{DenefVeys}. Namely, for each prime $p$ there exists a function
$L_{F,p}:\NN\to \NN$ with $L_{F,p}(m) \ll m^{n-1}$, such that for
all $m\geq 2$ and all $y\in\ZZ_p^n$, both estimates
\begin{equation}\label{fix1}
| S(F,p,m) |_\CC \leq L_{F,p}(m) p^{-m a(F)}
\end{equation}
and
\begin{equation}\label{fixp}
| S_y(F,p,m) |_\CC \leq L_{F,p}(m) p^{-m a_{y,p}(F)}
\end{equation}
hold.
In the case that $F$ is non-degenerate with respect to (the compact faces of) the Newton polyhedron at zero of $F$, then the bounds (\ref{1y}) with $y=0$ hold, see \cite{DenSper} and \cite{CDenSperlocal}. If $F$ is non-degenerate and quasi-homogeneous, then also the bounds from (\ref{1}) hold, by \cite{DenSper} and \cite{CDenSperlocal}.
For other work on Igusa's original conjecture, we refer to
\cite{CDenSper}, \cite{Cigumodp}, \cite{Lichtin4}, \cite{JWright}.
Lemma 5.4 of \cite{Birch} gives other evidence for Conjecture
\ref{new}, under some specific geometric conditions. Related
exponential sums in few variables (namely with small $n$) have been
studied in \cite{Lichtin4}, \cite{JWright} and in \cite{CochZheng1},
\cite{CochZheng2}.
 \\

Below we will write $|\cdot|$ instead of $|\cdot|_\CC$ for the
complex norm. For complex valued functions $H$ and $G$ on a set $Z$,
the notation $H\ll G$ means that there exists a constant $c>0$ such
that $|H(z)|\leq c |G(z)|$ for all $z$ in $Z$.
 All integrals over $K^n$, for any non-archimedean local field $K$
with valuation ring $\cO_K$, will be against the Haar measure $|dx|$
on $K^n$, normalized so that $\cO_K^n$ has measure $1$.
 We write $\FF_p^{\rm alg}$ for an algebraic closure of $\FF_p$, the field with $p$ elements.

\section{Proofs of the main results}\label{proofs}

We first recall two descriptions of the log-canonical threshold.

\begin{defn}\label{c(f)}
For a non-constant polynomial $f$ in $n$ variables over an
algebraically closed field $K$ of characteristic zero, and $y\in
K^n$ satisfying $f(y)=0$, the log-canonical threshold of $f$ at $y$
is denoted by $c_y(f)$ and defined as follows. For any proper
birational morphism $\pi:Y\to K^n$ from a smooth variety $Y$, and
for any prime divisor $E$ on $Y$, we denote by $N$ and $\nu-1$ the
multiplicities along $E$ of the divisors of $\pi^*f$ and
$\pi^*(dx_1\wedge\dots\wedge dx_n)$, respectively. Then
$$
c_y(f) = \inf_{\pi,E} \{\frac{\nu}N\} ,
$$
where $\pi$ runs over all $\pi$ as above and $E$ over all prime
divisors on $Y$ such that $y\in\pi(E)$. For a polynomial $f$ over a
non-algebraically closed field $k$ of characteristic zero and $y\in
k^n$ satisfying $f(y)=0$, one defines $c_y(f)$ as above with $K$ any
algebraic closure of $k$. Finally, when $f$ is the zero
polynomial, one defines $c(f)$ as $0$.

In fact $c_y(f)=\min_E \{\frac{\nu}N\}$, where $\pi$ is any fixed embedded resolution of the germ of $f=0$ at $y$ (and $y\in\pi(E)$).
Note that always $c_y(f)\leq 1$, a property not shared by the motivic oscillation index of $f$, and neither by the complex oscillation index of $f$, see \cite[Chapter 13, and, p. 203]{Arnold.G.V.II}, \cite{Malgrange}, \cite{Cigumodp}.
\end{defn}

By Musta\c t{\v{a}}'s Corollaries 0.2 and 3.6 of \cite{MustJAMS}, we can
describe the log-canonical threshold by taking certain codimensions,
as follows.

Let $p$ be an integer and $h$ a nonconstant polynomial over $\CC$ in $n$ variables. Write $\Cont^{\geq p}(h)$ for the subset of $\CC[[t]]^n$ given by
$$
\{x\in \CC[[t]]^n\mid  h(x)\equiv 0 \bmod (t^p)\}
$$
and $\Cont_0^{\geq p}(h)$ for
$$
\{x\in \CC[[t]]^n\mid \ord_t h(x)\equiv 0 \bmod (t^p),\, x\in (t\CC[[t]])^n\}.
$$
Let us further write
$$
\codim \Cont^{\geq p}(h)
$$
for the codimension of $\rho_m(\Cont^{\geq p}(h))$ in
$\rho_m(\CC[[t]]^n)$ for any $m\geq p$, where $\rho_m:\CC[[t]]^n\to
(\CC[t]/(t^{m+1}))^n$ is the projection modulo $t^{m+1}$ in each
coordinate. Here, $\rho_m(\Cont^{\geq p}(h))$ is seen as a Zariski
closed subset of $\CC^{n(m+1)}\cong \rho_m(\CC[[t]]^n)$. The
definition is independent of the choice of $m$. We write similarly
$\codim \Cont_0^{\geq p}(h)$ for the codimension of
$\rho_m(\Cont_0^{\geq p}(h))$ in $\rho_m(\CC[[t]]^n)$ for any $m\geq
p$.

By Corollary 0.2 of \cite{MustJAMS}, for all integers $k>0$, we have
\begin{equation}\label{=}
c(h) \leq \frac{\codim \Cont^{\geq k}(h)}{k}
\end{equation}
and there exist infinitely many $k>0$ for which equality holds.
Also, if $h$ vanishes at $0$, one has by Corollary 3.6 of \cite{MustJAMS} that
\begin{equation}\label{<=}
c_0(h) = \inf_{k>0} \frac{\codim \Cont_0^{\geq k}(h)}{k}.
\end{equation}
Based on these relations, we can now prove Proposition \ref{geomcfr}.

\begin{proof}[Proof of Proposition \ref{geomcfr}]
By the equality statement concerning (\ref{=}) for $f_r$, there
exists $k>0$ such that
\begin{equation}\label{=2}
c(f_r) = \frac{\codim \Cont^{\geq k}(f_r)}{k}.
\end{equation}
Let $\ell$ be $kr+k$.
Now define the cylinder $B\subset \CC[[t]]^n$ as
$$
B:=\{x\in \CC[[t]]^n\mid \rho_{k-1}(x)=\{0\},\ \ord_t f_r(x) \geq
\ell\}.
$$
By the homogeneity of $f_r$, the cylinder $B$ can be considered
(under corresponding identifications), as
$$
\rho_{k-1}(B)\times t^k\Cont^{\geq k}(f_r) = \{0\}\times t^k\Cont^{\geq k}(f_r)\subset \CC[[t]]^n.
$$
Again by the homogeneity of $f_r$ and the fact that $f-f_r$ has multiplicity at least $r+1$, one has
$$
B\subset \Cont_0^{\geq \ell}(f).
$$
Hence, by (\ref{<=}), one finds
\begin{equation}\label{dimB}
c_0(f) \leq   \frac{\codim B}{\ell},
\end{equation}
where $\codim B$ is defined as the codimension of $\rho_m(B)$ in $\rho_m(\CC[[t]]^n)$ for large enough $m$.
On the other hand, one finds from (\ref{=}) that
$$
\codim B =  kn + \codim( \Cont^{\geq k}(f_r) ) =  kn +  k c(f_r).
$$
Using this together with (\ref{dimB}) and dividing by $k$, one finds (\ref{c_0}).
\end{proof}

It is also possible to give a proof for Proposition \ref{geomcfr} based on embedded resolution of singularities,
without using Musta\c t{\v{a}}'s formulas.

\begin{proof}[Alternative proof of Proposition \ref{geomcfr}]
 Let $\pi_0:Y_0\to\CC^n$ be the blowing-up at the
origin; its exceptional divisor $E_0$ is projective $(n-1)$-space.
We consider for example the chart on $Y_0$ where $E_0$ is given by
$x_1=0$ and $\pi_0^*f$ by
$$
x_1^r\Big(f_r(1,x_2,\dots,x_n)+x_1\sum_{i\geq
r+1}x_1^{i-r-1}f_i(1,x_2,\dots,x_n)\Big).
$$
Along $E_0$ the multiplicity of the pullback of
$dx=dx_1\wedge\dots\wedge dx_n$ is $n$ and the multiplicities of
both $\pi_0^*f$ and $\pi_0^*f_r$ are $r$.

We now perform a composition of blowing-ups $Y\to Y_0$, leading to
an embedded resolution $\pi:Y\to \CC^n$ of $f_r=0$. More precisely,
for example on the chart above, we only use centres \lq not
involving $x_1$\rq; hence they all have positive dimension and are
transversal to $E_0$.
 Say
$c(f_r)=\frac{\nu}N$, where $E$ is an exceptional component of $\pi$
such that along $E$ the multiplicities of the pullback of $dx$ and
$f_r$ are $\nu$ and $N$, respectively. We may assume that $E\neq
E_0$; otherwise $c(f_r)=\frac nr$ and the statement becomes trivial.

Consider analytic or \'etale coordinates $x_1,y_2,\dots,y_n$ in a
generic point of $E\cap E_0 \subset Y$ such that $E$ is given by
$y_2=0$. In that point $\pi^*f$ is of the form
$$
x_1^r\Big(y_2^N u(y_2,\dots,y_n)+x_1(\dots)\Big),
$$
where $u(y_2,\dots,y_n)$ is a unit. Next, we blow up $Y$ at the
codimension two centre $Z_1=E\cap E_0$ given (locally) by
$x_1=y_2=0$. Along the new exceptional divisor $E_1$ the
multiplicities of the pullback of $dx$ and $f$ are $n+\nu$ and
$r+\mu_1$, respectively, where $\mu_1 \geq 1$ is the order of
vanishing of $y_2^N u(y_2,\dots,y_n)+x_1(\dots)$, the strict
transform of $f$, along $Z_1$. In fact, in the relevant chart the
pullback of $f$ is now of the form
$$
x_1^r y_2^{r+\mu_1}\Big(y_2^{N-\mu_1}
u(y_2,\dots,y_n)+x_1(\dots)\Big).
$$
As long as $E_0$ intersects the strict transform of $f=0$, we
continue to blow up with centre this intersection, in the relevant
chart always given by $x_1=y_2=0$. Let $E_k$ be the last exceptional
component created this way. Then along $E_k$ the multiplicities of
the pullback of $dx$ and $f$ are $kn+\nu$ and
$kr+\sum_{i=1}^k\mu_i$, respectively, where the $\mu_i$ are the
orders of vanishing of the strict transform of $f$ along the centres
of blow-up. Note that $\sum_{i=1}^k\mu_i=N$. We just showed that
\begin{equation}\label{bound1}
c_0(f) \leq \frac{kn+\nu}{kr+N}.
\end{equation}
An elementary computation, using that $\frac{\nu}N \leq \frac nr$
and $k\leq N$, shows that
\begin{equation}\label{bound2}
 \frac{kn+\nu}{kr+N}\leq \frac {n+\frac{\nu}N }{r+1}=\frac {n+c(f_r) }{r+1}.
\end{equation}
Then combining (\ref{bound1}) and (\ref{bound2}) finishes the proof.
\end{proof}

\begin{remark} (1) The proof above can be
shortened by using a weighted blow-up instead of the last $k$
blow-ups.

(2) M. Musta{\c{t}}{\v{a}} informed us of yet another proof of
Proposition \ref{geomcfr}, using multiplier ideals.
\end{remark}

\smallskip

\begin{proof}[Proof of Lemma \ref{geomc}]
The inequality (\ref{c_0-1}) follows from (\ref{c_0}) and
(\ref{trivial0}) for $f_r$. For inequality (\ref{c_0-2}) and with
$g$ as in the lemma, consider the cylinder $C$ given by
$$
\{x\in \CC[[t]]^n\mid \rho_{0}(x)=0,\ \ord_t
g(\frac{x_1}{t},\ldots,\frac{x_n}{t}) \geq 1 \}.
$$
Then one easily verifies that
$$
C\subset \Cont_0^{\geq r+2}(f)
$$
and $\codim C = n+ 1$. The result now follows from Musta\c t{\v{a}}'s
bound as in (\ref{<=}) for $f$ and $k=r+2$.
\end{proof}

\begin{proof}[Proof of Proposition \ref{geom2}]
By Theorem 1.2 of \cite{DimcaMST}, one has for generic $y$ in $Y$
and a generic vector subspace $H$ of $\CC^n$ of dimension $n-d$ that
$$
c_0(f_{y|H}) = c_{y}(f),
$$
where $f_{y|H}$ is the restriction of the polynomial map $f_y$ to
$H$. The proposition now follows from the genericity of $y$ and $H$,
by (\ref{trivial}) and by Proposition \ref{geomcfr} and Lemma
\ref{geomc} applied to $f_{y|H}$.
\end{proof}

In the proof of our main theorems we will use the following lemmas. The first one follows almost directly from work by Katz in \cite{Katz} and Noether normalization.
\begin{lem}\label{lemKa}
Let $n,k,N$ be nonnegative integers. Then there exist constants $D$ and $E$ such that the following hold for all prime numbers $p$ with $p>E$, all positive powers $q$ of $p$, and all nontrivial additive characters $\psi_q$ on $\FF_q$.
Let $g_1,\ldots, g_k$ and $h$ be (nonconstant) homogeneous polynomials in $x=(x_1,\ldots,x_n)$ with coefficients in $\ZZ$ and of degree at most $N$. Let $X$ be the reduced subscheme of $\AA^n_\ZZ$ associated to the ideal $(g_1,\ldots,g_k)$.

If $h$ (modulo $p$) does not vanish on any irreducible component of
$X_p:= X\otimes \FF_p^{\rm alg}$ of dimension equal to $\dim X_p$,
then
\begin{equation}\label{e1}
| \sum_{y\in X(\FF_q)}  \psi_q\big(h(y)\big) | \leq
D\cdot q^{ \dim X_p - 1/2 }.
\end{equation}

If the image of  $h$ in $\FF_p^{\rm alg}[x]$ under $\ZZ[x]\to \FF_p^{\rm alg}[x]$ is reduced, then
\begin{equation}\label{e2}
| \sum_{y\in  \FF_q^n}  \psi_q\big(h(y)\big) | \leq
D\cdot q^{ n  - 1}.
\end{equation}
\end{lem}
\begin{proof}
The bounds in (\ref{e2}) follow immediately from Katz \cite{Katz}, Theorem 4.
In the case that $X_p$ is irreducible, the bounds in (\ref{e1}) follow from Theorem 5 of \cite{Katz}. The remaining case that $X_p$ is reducible follows from the irreducible case and Noether normalization.
\end{proof}

From now on, let $F$ and $r$ 
be as in the introduction.
We will use some instances of the Ax-Kochen principle, Theorem 6 of \cite{AK1}, like the following lemma. 

\begin{lem}\label{AK2}
For large enough $p$, any $v\in \FF_p^n$, and any $y\in \ZZ_p^n$ lying above $v$, the following holds. If the reduction of $F$ modulo $p$ vanishes with order $r$ at $v$, then
$$
 \ord ( F(y) )\geq r,
$$
where $\ord$ is the $p$-adic order $\QQ_p\to\ZZ\cup\{+\infty\}$.
\end{lem}
\begin{proof}
The statement is easily reduced to a simple statement over a discrete valuation ring of equicharacteristic zero. One finishes by a standard ultraproduct argument (namely by the Ax-Kochen principle).
\end{proof}

\begin{lem}\label{smooth}
Let $V$ be the subscheme of $\AA_\ZZ^n$ given by the equations $\grad F=0$.
If $p$ is large enough, then one has for any $m>1$ that
$$
S(F,p,m)   =  \sum_{v\in V(\FF_p)} \int_{u\in\ZZ_p^n,\ u\equiv v\bmod p}  \exp ( 2\pi i \frac{F(u)}{p^m} )   |du|
$$
and that $S_y(F,p,m)=0$ whenever the reduction of $y$ modulo $p$ does not lie in $V(\FF_p)$.
\end{lem}
\begin{proof}
This follows by taking Taylor series around $y$ and by the basic relation
$$
\sum_{t\in\FF_p}\psi_p(t)=0
$$
for any nontrivial additive character $\psi_p$ on $\FF_p$.
\end{proof}

We begin with the proof of the almost trivial part of Theorem \ref{rplus1}.
\begin{proof}[Proof of Theorem \ref{rplus1} for $m\leq r$, resp.~$m\leq r_{y,p}$]
Note that for small $p$, there is nothing to prove by (\ref{fix1}), resp.~(\ref{fixp}). If $r=+\infty$, the theorem follows easily. We may thus suppose that $r<+\infty$ and that $p$ is large.
Let $V$ be the subscheme of $\AA_\ZZ^n$ given by the equations $\grad F=0$, and write $d$ for the dimension of $V\otimes \CC$.
Fix $m>1$ with $m\leq r$, resp.~$m\leq r_{y,p}$.
For all $y\in\ZZ^n$ one has
$$
ma(F) \leq r a(F) \leq n - d,
$$
by (\ref{c_0b}), resp.
$$
 ma_{y,p}(F)\leq r_{y,p} a_{y,p}(F).
$$
Also, when $p$ is large enough, one has
$$
S(F,p,m)  = p^{-n} \# V(\FF_p), 
$$
resp.,
\begin{equation}\label{rmn}
S_y(F,p,m) = p^{-n}\mbox{ and }r_{y,p} a_{y,p}(F) \leq n
\end{equation}
for $y\bmod p$ in $V(\FF_p)$,
 and $$S_y(F,p,m)  = 0$$
for $y\bmod p$ outside $V(\FF_p)$.
  Indeed, this follows by Lemmas \ref{AK2} and \ref{smooth}.
By Noether normalization there exists $D$ such that
$$
\# V(\FF_p) \leq D p^d,
$$
uniformly in $p$.
One readily finds
$$
|S(F,p,m)| \leq D p^{-ma(F)},
$$
resp.
$$
|S_y(F,p,m)| \leq p^{-ma_{y,p}(F)},
$$
for all large $p$ and all $y\in\ZZ^n$, which finishes the proof.
\end{proof}

\begin{proof}[Proof of Theorem \ref{rplus1} for $m=r+1$, resp.~$m= r_{y,p}+1$]
Note that for small $p$, there is nothing to prove by (\ref{fix1}), resp.~(\ref{fixp}). We may thus again suppose that $p$ is large and that $r<+\infty$.
Fix $y\in\ZZ^n$. By Lemma \ref{smooth} we may suppose that there
exists a critical point $y'\in y+\ZZ_p^n$ of $F$, such that
$F-F(y')$ vanishes with order $r_{y,p}$ at $y'$ and $c_{y'}(F-F(y')) =
a_{y,p}(F)$. Write $f_y(x)$ for $F(x+y')-F(y')$ and $f_y =
\sum_{i\geq r_{y,p}} f_{y,i}$ with $f_{y,i}$ either identically zero
or homogeneous and of degree $i$ and with $f_{y,r_{y,p}}$ nonzero
for a choice of such $y'$. We first prove (\ref{by}) by the
following calculation, where $\psi$ is the additive character on
$\QQ_p$ sending $x$ to $\exp(2\pi i x')$ for any rational number
$x'$ which lies in $\ZZ[1/p]$ and satisfying $x-x'\in \ZZ_p$, and
with Haar measure normalized as in section \ref{not}:
\begin{eqnarray*}
S_y (F,p,r_{y,p}+1) & = & \int_{x\in y+ (p\ZZ_p)^n} \psi\big(\frac{F(x)}{p^{r_{y,p}+1}}\big)|dx|\nonumber \\
& = & \int_{x\in (p\ZZ_p)^n} \psi\big(\frac{f_{y}(x)  + F(y') }{p^{r_{y,p}+1}}\big)|dx|\nonumber \\
& = & \frac{b_y}{p^n}\int_{u\in \ZZ_p^n} \psi\big(\frac{p^{r_{y,p}}f_{y,r_{y,p}}(u) + p^{r_{y,p}+1}f_{y,r_{y,p}+1}(u) +\ldots }{p^{r_{y,p}+1}}\big)|du|\\
& = & \frac{b_y}{p^n}\int_{u\in \ZZ_p^n} \psi\big(\frac{p^{r_{y,p}}f_{y,r_{y,p}}(u)}{p^{r_{y,p}+1}}\big)|du|\\
& = & \frac{b_y}{p^n}\int_{u\in \ZZ_p^n} \psi\big(\frac{f_{y,r_{y,p}}(u)}{p}\big)|du|\\
& = & \frac{b_y}{p^n}\sum_{v\in \FF_p^n} \int_{u\in\ZZ_p^n,\ \overline u=v} \psi\big(\frac{f_{y,r_{y,p}}(u)}{p}\big)|du|\\
& = & \frac{b_y}{p^{2n}} \sum_{v\in \FF_p^n} \psi_p\big(
\overline{f_{y,r_{y,p}}}(v) \big).
\end{eqnarray*}
Here we denote by $\overline u$ the tuple in $\FF_p^n$ obtained by
reduction mod $p$ of the components $u_i\in\ZZ_p$ of $u$, by
$\psi_p$ the nontrivial additive character on $\FF_p$ sending $w$ to
$\psi(w'/p)$ for any $w'\in\ZZ_p$ which projects to $w$, by
$\overline{f_{y,r_{y,p}}}$ the reduction modulo $p$ of
$f_{y,r_{y,p}}$, and we put
$$b_y:=\psi\Big(\frac
{F(y')}{p^{r_{y,p}+1}}\Big).$$
 Now by Lemma \ref{lemKa}, applied to
$h=f_{y,r_{y,p}}$ and with $k=0$, there exists a constant $D>0$ such
that
$$
|\sum_{v\in \FF_p^n}  \psi_p\big(\overline{f_{y,r_{y,p}}}(v)\big)| \leq
D\cdot p^{n-\delta_{y,p}}
$$
for each large $p$ and uniformly in $y$ for $\delta_{y,p}$ so that  $\delta_{y,p}=1/2$ in the case that $\overline{f_{y,r_{y,p}}}$ is non-reduced, and $\delta_{y,p}=1$  in the case that $\overline{f_{y,r_{y,p}}}$ is reduced.

We claim, for large $p$ and for all $y\in\ZZ^n$, 
that
\begin{equation}\label{red}
(r_{y,p}+1)c_0(f_y) \leq n+\delta_{y,p}.
\end{equation}
If $y'$ is a non-isolated critical point of $F$ (in the set of
critical points of $F$ with coordinates in an algebraic closure of
$\QQ_p$), then $r_{y,p}c_0(f_y) \leq n-1$ by (\ref{c_0b}) and the
claim follows from $c_0(f_y)\leq 1$. Also, if $\delta_{y,p}=1$, then
the claim follows from (\ref{trivial}) and $c_0(f_y)\leq 1$. In the
case that $y'$ is an isolated critical point (in the set of critical
points of $F$ with coordinates in an algebraic closure of $\QQ_p$)
and $\delta_{y,p}=1/2$ simultaneously, it follows from our
assumption that $p$ is large that $f_{y,r_{y,p}}$ is non-reduced and
thus (\ref{red}) follows from Lemma \ref{geomc}. This
assumption of $p$ being large is uniform in $y$ since there are only
finitely many isolated critical points of $F$. Hence, we find for
all large $p$ and all $y$ that
\begin{eqnarray}
| S_y(F,p,r_{y,p}+1) | & = & \frac{1}{p^{2n}} |\, \sum_{v\in \FF_p^n}  \psi_p\big( \overline{f_{y,r_{y,p}}}(v) \big) \ | \nonumber \\
 & \leq &  D\cdot  p^{-n - \delta_{y,p}} \label{y}\\
 & \leq &  D\cdot p^{ -(r_{y,p}+1) c_0(f_y) } \leq  D\cdot p^{ -(r_{y,p}+1) a_{y,p}(F) }.
 \end{eqnarray}
This completes the proof of (\ref{by}) for all $y$ and $m=r_{y,p}+1$.

\smallskip
To show (\ref{b}), let $V$ be the subscheme of $\AA_\ZZ^n$ given by
the equations $\grad F=0$, and let $d$ be the dimension of $V\otimes
\CC$. For each $v\in V(\FF_p)$, fix a point $y(v)$ in $\ZZ^n$ lying
above $v$, and a critical point $y'(v)$ of $F$ lying above $v$ such
that $F-F(y'(v))$ vanishes with order $r_{y(v),p}$ and $c_{y'}(F-F(y')) =
a_{y,p}(F)$ (such $y'$ exists since $p$ is assumed large). Now
(\ref{b}) for $m=r+1$ follows by estimating, for large primes $p$,
\begin{eqnarray}
| S(F,p,r+1) | & = & | \sum_{v\in V(\FF_p)} S_{y(v)} (F,p,r+1) | \label{a1}\\
&   \leq  &  \sum_{v\in V(\FF_p)} D\cdot  p^{-n - \varepsilon_v},  \label{a2}
     \end{eqnarray}
for some $D>0$, and where $\varepsilon_v $ equals
$\delta_{y(v),p}(y'(v))$ whenever $r= r_{y,p}$ and where
$\varepsilon_v = 0$ when $r< r_{y,p}$. Here the equality (\ref{a1})
follows from Lemma \ref{AK2}, and the inequality (\ref{a2}) comes
from (\ref{y}) when $r= r_{y,p}$ and from (\ref{rmn}) when $r<
r_{y,p}$.
By quantifier elimination for the language of rings with coefficients in $\ZZ$, there exist $V_0$, $V_{1/2}$, and $V_1$, such that $V_i$ is a finite disjoint union of subschemes
of $V$ (it is constructible and defined over $\ZZ$) with $\cup_i V_i(\CC) = V(\CC)$ and such that the following hold, for $i=0$, $\frac{1}{2}$, and $1$. The polynomial $F-F(b)$ vanishes with order $>r$ at $b$ for $b\in
V_0(\CC)$, $F-F(b)$ vanishes with order $r$ at $b$ for $b\in
V_{1/2}(\CC)$ and also for $b\in V_{1}(\CC)$, and $(F(x+b)-F(b))_r$
is reduced for $b\in V_{1}(\CC)$, and non-reduced for $b\in
V_{1/2}(\CC)$. Let $d_i$ be the dimension of
$V_i\otimes \CC$. Note that for large $p$, one has $\varepsilon_v = i $ for $v\in V_i(\FF_p)$.  
 Now we bound as follows:
\begin{eqnarray}
| S(F,p,r+1) | & \leq & \sum_i \#V_i(\FF_p) D\cdot  p^{-n - i }\label{a300} \\
& \leq &   \sum_i \#V_i(\FF_p) \cdot D\cdot p^{ -(r+1) a(F) - d_i   }  \label{a30} \\ 
& \leq & D' p^{-ma(F)} \label{a4},
     \end{eqnarray}
for some $D'$. The inequality (\ref{a300}) follows from (\ref{a2}),
(\ref{a30}) follows from Proposition \ref{geom2} and the definition of
$a(F)$ as a minimum, and (\ref{a4}) from Noether normalization.
\end{proof}

\begin{proof}[Proof of Theorem \ref{rplus1} for $m=r+2$, resp.~$m= r_{y,p}+2$.]
For the same reasons as in the previous proofs we may concentrate on
large primes $p$ and suppose $r<+\infty$. Fix $y\in\ZZ^n$.
By Lemma \ref{smooth} we may suppose that there
exists a critical point $y'\in y+\ZZ_p^n$ of $F$, such that
$F-F(y')$ vanishes with order $r_{y,p}$ at $y'$ and $c_{y'}(F-F(y')) =
a_{y,p}(F)$. Write $f_y(x)$
for $F(x+y')-F(y')$ and $f_y = \sum_{i\geq r_{y,p}} f_{y,i}$ with
$f_{y,i}$ either identically zero or homogeneous and of degree $i$,
and where $f_{y,r_{y,p}}$ is nonzero. We first prove (\ref{by}). Let
$X$ be the subscheme of $\AA_{\ZZ_p}^n$ associated to the equations $\grad
f_{y,r_{y,p}} =0$. Let $A_p$ be the subset of $\ZZ_p^n$ of those
points whose projection mod $p$ lies in $X(\FF_p)$. Also, let $C_p$
be the complement of $A_p$ in $\ZZ_p^n$.
We calculate as follows:  
\begin{eqnarray*}
S_y(F,p,r_{y,p}+2)
& = & \int_{x\in y+ (p\ZZ_p)^n} \psi\big(\frac{F(x)}{p^{r_{y,p}+2}}\big)|dx|\nonumber \\
& = & \frac{b_y}{p^n}\int_{u\in \ZZ_p^n} \psi\big(\frac{p^{r_{y,p}}f_{y,r_{y,p}}(u) + p^{r_{y,p}+1}f_{y,r_{y,p}+1}(u) 
 }{p^{r_{y,p}+2}}\big)|du|\\
& = & \frac{b_y}{p^n}\int_{u\in \ZZ_p^n} \psi\big(\frac{f_{y,r_{y,p}}(u) + p f_{y,r_{y,p}+1}(u)  }{p^{2}}\big)|du|\\
& = & \frac{b_y}{p^n} \big( I_1 + I_2 \big),
\end{eqnarray*}
where $b_y = \psi\Big(\frac{F(y')}{p^{r_{y,p}+2}}\Big)$,
$$
I_1 = I_{1}(y) = \int_{u\in A_p} \psi\big(\frac{f_{y,r_{y,p}}(u) + pf_{y,r_{y,p}+1}(u)  }{p^{2}}\big)|du|,
$$
and
$$
I_2 = I_{2}(y) = \int_{u\in C_p} \psi\big(\frac{f_{y,r_{y,p}}(u) + pf_{y,r_{y,p}+1}(u)  }{p^{2}}\big)|du|.
$$
One has $I_2 = 0$ by Hensel's Lemma and by the basic relation
$$
\sum_{t\in\FF_p}\psi_p(t)=0
$$
for the nontrivial additive character $\psi_p$ on $\FF_p$.

\par
To estimate $|I_1|$, we first assume the condition on $y$ and $y'$
that $f_{y,r_{y,p}+1}$ vanishes on at least one absolutely
irreducible component of $X$ of maximal dimension. We will show that
this condition on $y$ and $y'$ implies
\begin{equation}\label{van}
(r_{y,p}+2)c_0(f_y) \leq 2n- \dim (X\otimes \QQ_p).
\end{equation}
If $\dim (X\otimes \QQ_p) \leq n-2$, then (\ref{van}) follows from
$(r_{y,p}+2)c_0(f_y) \leq n+2$, which in turn follows from
$c_0(f_y)\leq 1$ and (\ref{trivial}). If $\dim X\otimes \QQ_p =n-1$
one has that $(r_{y,p}+2)c_0(f_y) \leq n +1$ by Lemma \ref{geomc},
and (\ref{van}) follows also in this case and thus in general.
By Noether normalization, there exists $E>0$ independent of $y$ such that
$$
\# X(\FF_p) \leq E p^{\dim (X\otimes \QQ_p)}
$$
for all large $p$.
Since
$$
|I_1| \leq  \frac{\# X(\FF_p)}{p^{n}},
$$
we find from the above discussion that, for all $y$ satisfying the above condition,
$$
\frac{1}{p^n}|I_1| \leq E p^{\dim (X\otimes \QQ_p) -2n} \leq E
p^{-(r_{y,p}+2)c_0(f_y)} \leq E p^{-(r_{y,p}+2) a_{y,p}(F)}
$$
for all large $p$.

\par
Finally assume the condition on $y$ and $y'$ that $f_{y,r_{y,p}+1}$
does not vanish on any absolutely irreducible component of $X$ of
maximal dimension. By Lemma \ref{AK2}, one can rewrite $I_1$ for
large $p$ as
$$
I_1 = \int_{u\in A_p} \psi\big(\frac{ f_{y,r_{y,p}+1}(u)  }{p}\big)|du|.
$$
Using this expression we compute
\begin{eqnarray*}
\frac{1}{p^n} I_1
& = & \frac{1}{p^n}\sum_{v\in X(\FF_p)} \int_{\overline u=v,\ u\in \ZZ_p^n} \psi\big(\frac{f_{y,r_{y,p}+1}(u)}{p}\big)|du|\\
& = & \frac{1}{p^{2n}} \sum_{v\in X(\FF_p)}  \psi_p\big( \overline{f_{y,r_{y,p}+1}}(v) \big),
\end{eqnarray*}
where the notations $\overline u$, $\psi_p$, and $\overline{f_{y,r_{y,p}+1}}$ are as in the proof of the case $m=r_{y,p}+1$, namely reductions modulo $p$.
By Lemma \ref{lemKa}, there exists $N>0$ such that, for all $y$ satisfying the above condition, and for all large $p$,
$$
|\sum_{y\in X(\FF_p)}  \psi_p\big( \overline{f_{y,r_{y,p}+1}}(y) \big)| \leq N p^{\dim (X\otimes \QQ_p) - 1/2}.
$$
Hence,
$$
|\frac{1}{p^n} I_1| \leq N p^{-2n + \dim (X\otimes \QQ_p) - 1/2}
$$
for large $p$. If $f_{y,r_{y,p}}$ is non-reduced, then $\dim
X\otimes \QQ_p = n-1$. If $f_{y,r_{y,p}}$ is reduced, then $\dim
(X\otimes \QQ_p)\leq n-2$. By (\ref{c_0-1}) of Lemma \ref{geomc},
$c_0(f_{y})\leq 1$ and (\ref{trivial}), one finds in any case that
$$
(r_{y,p}+2)c_0(f_y) \leq 2n - \dim (X\otimes \QQ_p) + 1/2.
$$
Hence,
$$
\frac{1}{p^n}|I_1| \leq N p^{- (r_{y,p}+2)c_0(f_y)} \leq N p^{-
(r_{y,p}+2)a_{y,p}(F)} = N p^{- ma_{y,p}(F)}
$$
for each large $p$, which finishes the proof of (\ref{by}) for $m=r_{y,p}+2$. One derives (\ref{b}) for $m=r+2$  by adapting the argument showing (\ref{by}) as in the proof for $m=r+1$.
\end{proof}

\subsection{Finite field extensions}\label{sec:K}
As usual it is possible to prove analogous uniform bounds for all finite field extensions of $\QQ_p$ and all fields $\FF_q((t))$, when one restricts to large residue field characteristics.
We just give the definitions and formulate the analogue of Conjecture \ref{new} and the analogue of Theorem \ref{rplus1}.

Let  $\cO$ be a ring of integers of a number field, and let $N>0$ be an integer. Let $F$ be a polynomial with coefficients in $\cO[1/N]$ in the variables $x=(x_1,\ldots,x_n)$.
Let $\CO[{[1/N]}]$ be the collection of all non-archimedean local fields $K$ (of any characteristic) with a ring homomorphism $\cO[1/N]\to K$ (where local means locally compact).
For $K$ in $\CO[{[1/N]}]$, write $\cO_K$ for its valuation ring with maximal ideal $\cM_K$
and residue field $k_K$ with $q_K$ elements.
Further write $\psi_K:K\to\CC^\times$ for an additive character
which is trivial on the valuation ring
$\cO_K$ and nontrivial on $\pi_K^{-1} \cO_K$ where $\pi_K$ is a uniformizer of $\cO_K$.
The analogue of the above integrals $S(F,p,m)$ and $S_y(F,p,m)$ for $K$ in $\CO[{[1/N]}]$ are the following integrals for $\lambda$ in $K^\times$,
$$
S(F,K,\lambda):=\int_{x\in \cO_K^n} \psi_K\big( \frac{F(x)}{\lambda}\big)|dx|
$$
and, for $y\in\cO_K^n$,
$$
S_y(F,K,\lambda):=\int_{x\in y + (\cM_K)^n} \psi_K\big(\frac{F(x)}{\lambda}\big)|dx|,
$$
where $|dx|$ is the Haar measure on $K^n$, normalized such that $\cO_K^n$ has measure one, and where $y + (\cM_K)^n = \prod_{i=1}^{n} (y_i+\cM_K)$.

The following naturally generalizes Conjecture \ref{new}, again formulated with the log-canonical threshold in the exponent, where other exponents, like the motivic oscillation index of \cite{Cigumodp} or the complex oscillation index of \cite[Section 13.1.5]{Arnold.G.V.II} or \cite{Malgrange}, that can be larger than $1$, again would make sense as well.

\begin{conj}\label{ConDSK}
There exist $M>0$ and a function $L_F:\NN\to \NN$ with $L_F(m) \ll  m^{n-1} $ such that for all $K\in \CO[{[1/N]} ]$ whose residue field has characteristic at least $M$, all $y\in\cO_K^n$, and all $\lambda\in K^\times$ with $ \ord (\lambda) \geq 2$, if one writes $m=\ord(\lambda)$, one has
$$
| S(F,K,\lambda) |_\CC \leq L_F(m) q_K^{-m a(F)},
$$
and
$$
| S_y(F,K,\lambda) |_\CC \leq L_F(m) q_K^{-ma_{y,K}(F)}.
$$
Here $\ord$ denotes the valuation on $K^\times$ with
$\ord(\pi_K)=1$, and $a_{y,K}(F)$ equals the minimum of the
log-canonical thresholds of $F(x)-F(y')$ at $y'$, where the minimum
is taken over all $y'\in y + (\cM_K)^n$.
\end{conj}

With the same proof as for Theorem \ref{rplus1},
we find the following.
\begin{thm}\label{rplus1K}
Let $F$ be a polynomial over $\cO[1/N]$.
There exist $M>0$ and a constant $L_F$ such that for all $K\in \CO[{[1/N]}]$ whose residue field has characteristic at least $M$ and for all $\lambda\in K^\times$, if one writes $m=\ord(\lambda)$ and if $2\leq m \leq r+2$, resp.~$2\leq m \leq r_{y,K}+2$, then one has
$$
|S(F,K,\lambda) |_\CC \leq L_F  q_K^{-m a(F)},
$$
resp.
$$
| S_y(F,K,\lambda) |_\CC \leq L_F q_K^{-ma_{y,K}(F)}.
$$
Here, $r_{y,K}(F)$ is the minimum of the order of vanishing of
$x\mapsto F(x)-F(y')$ where $x$ runs over those singular points of
the polynomial mapping $x\mapsto F(x)-F(y') :y+ \prod_{i=1}^{n}
(y_i+\cM_K) \to K $ with $c_{y'}(F-F(y')) = a_{y,K}(F)$.
\end{thm}

\subsection{A recursive bound for $c_0(f)$}\label{recursive}

We conclude the paper with a generalization of the bound of
Proposition \ref{geomcfr}, which also sharpens (\ref{c_0-2}). Let
$f$ be a non-constant polynomial over $\CC$ in the variables
$x=(x_1,\ldots,x_n)$ with $f(0)=0$, and write
\begin{equation}\label{ffi}
f = \sum_{i\geq 1} f_i,
\end{equation}
with $f_i$ either identically zero or homogeneous of degree $i$.

For $e$ a positive integer, let $d_e $ be the least common multiple
of the integers $1,2,\ldots,e $, and let $I_e(f) $ be the ideal
generated by the polynomials
$$
f_i^{d_e/(e -i+1)}
$$
for $i$ with $1\leq i\leq e $. Write $c(I_e(f))$ for the
log-canonical threshold of the ideal $I_e(f)$. (The log canonical
threshold $c(I)$ of a non-zero ideal $I$ in $n$ variables over $\CC$
can be defined analogously as in Definition \ref{c(f)}, for instance
as $\min_E \{\frac{\nu}N\}$, where $\pi$ is now any fixed
log-principalization of $I$ and $N$ is now the multiplicity along
$E$ of the divisor of $I\mathcal{O}_Y$. See e.g. \cite{MustJAMS} for
more details.) We put $c(I)=0$ when $I$ is the zero ideal.

\begin{thm}\label{geomcfe}
One has for any $e >0$ that
\begin{equation}\label{c_0e}
 (e +1)c_0(f) \leq n + d_e \cdot c(I_e(f) ).
\end{equation}
\end{thm}

Before proving Theorem \ref{geomcfe}, we state an equivalent
formulation and give some illustrative examples of (\ref{c_0e}).

Write as usual $f = \sum_{i\geq r} f_i$, where $f_r$ is nonzero. For
$k$ a positive integer, let $J_k(f)$ be the ideal generated by the
polynomials
$$
f_{r+i}^{d_k/(k -i)}
$$
for $i$ with $0\leq i\leq k-1$. Then
\begin{equation}\label{c_0k}
(r+k)c_0(f) \leq n + d_k \cdot c(J_k(f)).
\end{equation}
This reformulation (\ref{c_0k}) follows directly from (\ref{c_0e}),
using the multiplicativity of the log-canonical threshold, namely,
that $a \cdot c(I^a)= c(I)$ for any integer $a>0$ and any ideal $I$.
Its advantage is that the involved numbers are smaller.

For $k=1$, we obtain
$$
(r+1) c_0(f) \leq n + c(f_r),
$$
which is Proposition \ref{geomcfr}. The case $k=2$ sharpens and
generalizes (\ref{c_0-2}):
$$
(r+2)c_0(f) \leq n + 2c(f_r,f_{r+1}^2).
$$
As a third example, for $k=3$, we have
$$
(r+3)c_0(f) \leq n + 6c(f_r^2,f_{r+1}^3,f_{r+2}^6).
$$

\smallskip
The proof of Theorem \ref{geomcfe} is similar to the first one of
Proposition \ref{geomcfr}.

\begin{proof}[Proof of Theorem \ref{geomcfe}]
For any ideal $I$ of $\CC[x]$ and any integer $p>0$, we will write $\Cont^{\geq p}(I) $ for
$$
\{x\in \CC[[t]]^n\mid \ord_t h(x)\equiv 0 \bmod (t^p),\mbox{ for all }h\in I\}.
$$
By Corollary 3.4 of \cite{MustJAMS}, there exists $k>0$ such that
\begin{equation}\label{=e}
d_e k c(I_e(f)) = \codim \Cont^{\geq d_e k}(I_e(f)),
\end{equation}
where the codimension is taken as before (namely after projecting by
$\rho_m$ for high enough $m$).
Now define the cylinder $B\subset \CC[[t]]^n$ with $\rho_{k-1}(B)=\rho_{k-1}(\{0\})=\{0\}$
and,
(under corresponding identifications)
$$
B := \rho_{k-1}(\{0\})\times t^k\Cont^{\geq d_e k}(I_e(f)) = \{0\}\times t^k\Cont^{\geq d_e k}(I_e(f))\subset \CC[[t]]^n.
$$
By the homogeneity of the $f_i$, 
one checks for each $i$ that
$$
B\subset \Cont_0^{\geq k(e+1)}(f_i),
$$
and we thus have that
$$
B\subset \Cont_0^{\geq k(e+1)}(f).
$$
Hence, by Corollary 3.6 of \cite{MustJAMS}, one finds
\begin{equation}\label{dimBe}
k(e+1) c_0(f) \leq   \codim B.
\end{equation}
On the other hand, one finds by (\ref{=e}) and the definition of $B$ that
$$
\codim B =  k n + \codim( \Cont^{\geq d_e k}(I_e(f)) ) =  kn +  d_e k c(I_e(f)).
$$
Using this together with (\ref{dimBe}) and dividing by $k$, one finds (\ref{c_0e}).
\end{proof}

\begin{remark}
Also for Theorem \ref{geomcfe}, we could give another proof along
the lines of the alternative proof of Proposition \ref{geomcfr}.
More precisely, one blows up the origin, constructs a
log-principalization of the ideal $I_e(f)$, and performs an adequate
weighted blow-up in order to obtain an exceptional component with
the desired numerical invariants.
\end{remark}

\subsection*
{Acknowledgments} \hspace{0.5cm} The authors would like to thank
J.~Denef, E. Kowalski, and S.~Sperber for their generous sharing of insights. We
would like to express special thanks towards M. Musta{\c{t}}{\v{a}}
and J. Nicaise for discussions around Proposition \ref{geomcfr} and
Theorem \ref{geomcfe}, after which M. Musta{\c{t}}{\v{a}} showed us another proof for Proposition \ref{geomcfr}, different to the two given proofs above.

{\small{The first author would like to thank the Forschungsinstitut f\"ur Mathematik (FIM) at ETH Z\"urich and the MSRI in Berkeley, California, under Grant No. 0932078 000, for their hospitality during the spring semester of 2014.
The authors were supported in part by the European Research
Council under the European Community's Seventh Framework Programme
(FP7/2007-2013) with ERC Grant Agreements nr. 615722
MOTMELSUM, by the Labex CEMPI (ANR-11-LABX-0007-01), and by
the Research Fund KU Leuven (grant OT11/069).}}

\bibliographystyle{amsplain}
\bibliography{anbib}
\end{document}